\documentclass[twoside,leqno,12pt]{amsart}
\usepackage{amsfonts}
\usepackage{amsmath}
\usepackage{amscd}
\usepackage{amssymb}
\usepackage{amsthm, enumerate}
\usepackage{graphicx}
\usepackage{latexsym}
\usepackage{caption}
\usepackage{subcaption}

\usepackage{color}
\newtheorem{theorem}{Theorem}
\newtheorem{lemma}{Lemma}

\numberwithin{equation}{section}

\title{ A conditional approach for monochromatic unit
distance in a plane for four and five coloring }
\author{Saayan Mukherjee} 

\address{Saayan Mukherjee,
Ramakrishna Mission Vivekananda Educational and Research Institute, Department of Mathematics, Belur Math, Howrah, West Bengal 711202, India}
\email{saayanwith@gmail.com}

\begin{document}
\begin{abstract}
A measure theoretic approach of the problem that there exits a finite unit-distance graphs in the plane that are not five (or four) colorable.
\end{abstract}
\maketitle
\section{Introduction}
What is the minimum number of colors that are
required to color the points of the plane so that no two points of unit distance apart are assigned the same color? This number is called chromatic number of the plane, and it is usually denoted by $\chi(\mathbb{E}^2)$. This interesting question was first raised by Edward Nelson in 1950.\\
Leo and William Moser\cite{E} and later, Golomb \cite{C} constructed unit-distance graphs which proved that $\chi(\mathbb{E}^2)\geq 4 $ in figure 2.Figure 2 is called Moser Spindle which is the 7-node unit-distance graph.Hadwiger showed that $\chi(\mathbb{E}^2)\leq 7$ \cite{D} in figure 1.After a long time, Aubrey de Grey \cite{B} and later on Geoffrey Exoo and Dan Ismailescu \cite{A} constructed a unit distance graph $S_5$ thus improving the lower bound by 1 that is $\chi(\mathbb{E}^2)\geq 5$.

\begin{figure}[h]
\centering
    \begin{subfigure}{0.5\textwidth}
    \centering
    \includegraphics[width=0.8888\linewidth, height=6cm]{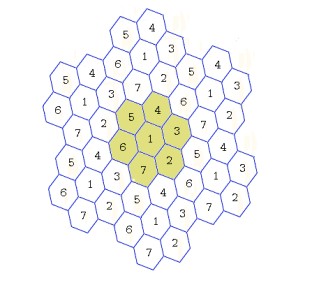}
   \caption{Figure 1. A 7-colouring of a tessellation of the plane by regular hexagons, with diameter slightly less than one. }
    \label{fig:my_label}
    \end{subfigure}%
  \begin{subfigure}{0.5\textwidth}
  \centering
      \includegraphics[width=0.8888\linewidth, height=6cm]{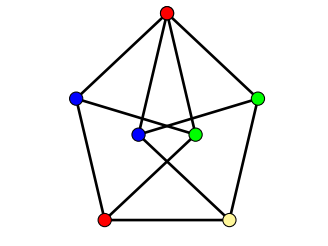}
    \caption{Figure 2. The Moser spindle}
    \label{fig:my_label}
    \end{subfigure}
\end{figure}

The idea of Geoffrey Exoo  and Dan Ismailescu is quite interesting.First they proved the following assertions.\\
(1)For any proper four coloring of $\mathbb{R}^2$ there exist a monochromatic pair of points distance $\sqrt{11/3}$ apart.\\
(2)If assertion 1 holds, then there exists a monochromatic equilateral triangle with side length $1/\sqrt{3}$.\\
(3)Let A, B and C be the vertices of an equilateral triangle of side $1/\sqrt{3}$. Then, there exists a unit distance graph containing A, B and C among its vertices, which cannot be $4$-colored under the restriction that A, B and C are identically colored.\\
Using these assertions they constructed a family of finite unit-distance graphs in the plane that are not four colorable.\\
In the rest of this paper, I will discuss this problem under an additional measure theoretic condition.This allows us to establish the existence of a monochromatic unit distance in a five (or four) coloring, provided this additional condition is satisfied.  

\section{Main Result}
\begin{theorem}

Suppose we have a five (or four) coloring of $\mathbb{R}^2$ where one of the five (or four) monochromatic sets is a set of Lebesgue measure $0$.Then there is a monochromatic unit distance in  $\mathbb{R}^2$.
\end{theorem}
To prove this we need to prove the following lemma.

\begin{lemma}
Let A be a subset of $\mathbb{R}^2$ of measure $0$.Let $\{p_1,p_2,...,p_n\}$ be any distinct points in $\mathbb{R}^2$.Then for any $\epsilon>0$ there is $\lambda\in[0,\epsilon]\times[0,\epsilon]$ such that $p_1,p_2,....,p_n \in \mathbb{R}^2 \setminus A$.In particular,there is $\lambda \in \mathbb{R}^2$ such that  $ \{p_1,p_2,...,p_n\} \in \mathbb{R}^2 \setminus A$.
\end{lemma}

\begin{proof}
Now we want to show that for any $\epsilon(>0)$, there exists $\lambda \in [0,\epsilon]\times[0,\epsilon]$ such that shifts of the points $ \{p_1,p_2,...,p_n\}$ by $\lambda$\\ satisfy $$\{p_1+\lambda,p_2+\lambda,....,p_n+\lambda\} \in R^2\setminus A.$$
If this is not the case,
then consider the set $S$ as follows,\\
$$S = \{\lambda\in[0,\epsilon]\times[0,\epsilon] : p_1+\lambda \in A~~~ \text{or}~~~ p_2+\lambda\in A~~~  \text{or} ....~~\text{or} ~~p_n+\lambda \in A\} = [0,\epsilon]\times[0,\epsilon].$$
So, the Lebesgue measure of the set $S$ is $\epsilon^2>0$.
On the other hand the set $S$ can be written as 
\begin{align*}
    S=&\cup_{j=1}^{n} \{\lambda\in[0,\epsilon]\times[0,\epsilon]: p_j+\lambda\in A\} \\
    =&\cup_{j=1}^{n} ((A-p_j) \cap ([0,\epsilon]\times[0,\epsilon]),where~ A-x=\{a-x:a\in A\}~for~x\in \mathbb{R}^2.
\end{align*}
But as $\mu(A)=0$,we have $\mu(A-p_j)=0.$
So, $\mu((A-p_j) \cap ([0,\epsilon]\times[0,\epsilon]))=0$.
Hence, $\mu(S)=0$, which is a contradiction.
\\
We note that the above lemma holds also when we replace the finite set $ \{p_1,p_2,...,p_n\}$ of points by a countable subset of $\mathbb{R}^2$, with essentially the same proof, but we need it only in its above version for finitely many points.
\end{proof}

\section{Proof of the Theorem} 
As we know that there exist a unit-distance graphs $S_5$ in the plane
that are not 4-colorable which is constructed at the url \cite{B}.Let $\{p_1,p_2,....,p_k\}$ is the vertices of the graph $S_5$. Let A be a set of Lebesgue measure 0, representing a monochromatic set.Theorem 1 follows if there exists $\lambda \in \mathbb{R}^2$ such that \{$p_1+\lambda,p_2+\lambda,....,p_{k}+\lambda\} \in \mathbb{R}^2 \setminus A$ i.e, for a suitable $\lambda \in \mathbb{R}^2,$ the vertices of the graph, shifted by $\lambda,$ lie all in $\mathbb{R}^2 \setminus A.$This is the case due to Lemma 1.
\\
Hence, the theorem for five coloring.
\\
In similar way, we can prove for four coloring also.
\section{conclusion}
Lemma 1 seems surprising on the first look since A can be even in dense in $\mathbb{R}^2$(take for example, $\mathbb{Q}^2$).Neverthless, its proof tells us that a suitable shift $\lambda$ always exist for which all of $\{p_1+\lambda,p_2+\lambda,....,p_n+\lambda\}$ lies in $\mathbb{R}^2 \setminus A$.However, the proof is non-constructive, i.e it doesn't yield a concrete $\lambda$.

\section{Acknowledgements}
 I would like to thank Prof. Stephan Baier (from RKMVERI) for helping me throught out the paper and also thanks to Prof. Malabika Pramanik (from University of British Columbia) for useful discussion.
 \\
I would also like to thank Prof. S.D. Adhikari (from RKMVERI) for his excellent course on combinatorics and his guidance to me.

\end{document}